\documentclass[12pt]{amsart}
\usepackage{amsmath, amssymb}
\title{Bernstein inequality and holonomic modules}
\author{Ivan Losev}
\address{I.L.: Department
of Mathematics, Northeastern University, Boston, MA 02115, USA}
\email{i.loseu@neu.edu}
\address{P.E.: Department of Mathematics, Massachusetts Institute of Technology,
Cambridge, MA 02139, USA}
\email{etingof@math.mit.edu}
\newcommand{\A}{\mathcal{A}}
\newcommand{\K}{\mathbb{C}}
\newcommand{\Z}{\mathbb{Z}}
\newcommand{\I}{\mathcal{I}}
\newcommand{\m}{\mathfrak{m}}
\newcommand{\HC}{\operatorname{HC}}
\newcommand{\VA}{\operatorname{V}}
\newcommand{\gr}{\operatorname{gr}}
\newcommand{\Ann}{\operatorname{Ann}}
\newcommand{\Weyl}{\mathbf{A}}
\newcommand{\Hom}{\operatorname{Hom}}
\newcommand{\B}{\mathcal{B}}

\newcommand{\g}{\mathfrak{g}}
\newcommand{\h}{\mathfrak{h}}
\newcommand{\C}{\mathbb{C}}
\newcommand{\Loc}{\operatorname{Loc}}
\newcommand{\param}{\mathfrak{p}}
\newcommand{\red}{/\!/\!/}
\newcommand{\quo}{/\!/}
\newcommand{\J}{\mathcal{J}}
\newcommand{\Leaf}{\mathcal{L}}
\newtheorem{Thm}{Theorem}[section]
\newtheorem{Prop}[Thm]{Proposition}
\newtheorem{Cor}[Thm]{Corollary}
\newtheorem{Lem}[Thm]{Lemma}
\theoremstyle{definition}

\newtheorem{Rem}[Thm]{Remark}

\begin{document}
\begin{abstract}
In this paper we study the representation theory of filtered algebras with
commutative associated graded whose spectrum has finitely many symplectic leaves. 
Examples are provided by the algebras of global sections of quantizations of symplectic
resolutions, quantum Hamiltonian reductions, spherical symplectic reflection algebras. 
We introduce the notion of holonomic modules for such algebras. We show that,
provided the algebraic fundamental groups of all leaves are finite, the 
generalized Bernstein inequality holds for the simple modules and turns into equality 
for holonomic simples. Under the same finiteness assumption, we prove that
the associated variety of a simple holonomic module is equi-dimensional.
We also prove that, if the regular bimodule has finite length,
then any holonomic module has finite length. This allows to reduce
the Bernstein inequality for arbitrary modules to simple ones.
We prove that the regular bimodule has finite length for the global
sections of quantizations of symplectic resolutions, for quantum Hamiltonian
reductions and for Rational Cherednik algebras.  The paper contains a joint appendix
by the author and Etingof that motivates the definition of a holonomic module 
in the case of global sections of a quantization of a symplectic resolution.
\end{abstract}
\maketitle
\section{Introduction}
\subsection{Holonomic modules}
Let $\A$ be an  infinite dimensional algebra. The study of all representations of $\A$ is a wild task even
if $\A$ itself is relatively easy, e.g., $\A$ is the algebra $D(\C^n)$ of algebraic linear differential
operators on $\C^{n}$ (a.k.a. the Weyl algebra of  the symplectic vector space $\C^{2n}$). In this case
(and in a more general case of the algebra $D(X)$ of differential operators on a smooth affine variety
$X$ or a sheaf $D_X$ of differential operators on a non-necessarily affine smooth variety) there is a nice
class of modules called {\it holonomic}. Namely, to a finitely generated $D_X$-module one can assign the so called
singular support that is a closed subvariety of $T^*X$. This subvariety is always coisotropic.
A $D_X$-module is called {\it holonomic} if its singular support is lagrangian. Many $D_X$-modules
appearing in ``nature'' are holonomic.

One question addressed in the present paper  is how to generalize the notion of a holonomic
module to a wider class of algebras. In the $D$-module situation it is easy to see
that any holonomic module has finite length. We will see that this is the case in
many other situations.

\subsection{Bernstein inequality}\label{SS_Bernstein_ineq}
Let $\g$ be a semisimple Lie algebra over an algebraically closed field of characteristic $0$.
We can form the universal enveloping algebra $U(\g)$. Let $\h$ be a Cartan subalgebra
of $\g$ and $W$ be the Weyl group. Then the center $Z$ of $U(\g)$ is identified
with $S(\h)^W$ via the so called Harish-Chandra isomorphism. So, for $\lambda\in \h^*$,
one can consider the central reduction $U_\lambda$ of $U(\g)$, i.e., the quotient
$U_\lambda:=U(\g)/U(\g)S(\h)^W_\lambda$, where $S(\h)^W_\lambda$ stands for the maximal ideal in
$S(\h)^W$ of all elements that vanish at $\lambda$. Then there is a remarkable property of
$U_\lambda$ known as the (generalized) Bernstein inequality: $$\operatorname{GK-}\dim(M)\geqslant
\frac{1}{2} \operatorname{GK-}\dim \left(U_\lambda/\operatorname{Ann}_{U_\lambda}(M)\right),$$
where $\operatorname{GK-}\dim$ stands for the Gelfand-Kirillov (shortly, GK) dimension and $\operatorname{Ann}_{U_\lambda}(M)$
denotes the annihilator of $M$ in $U_\lambda$, a two-sided ideal.

Another goal of this paper is to generalize the Bernstein inequality for a wider class of algebras.
We will also see  that the inequality becomes an equality for holonomic modules.

\subsection{Algebras of interest}\label{SS_alg}
Our base field is the field $\K$ of complex numbers. Let $\A=\bigcup_{i\geqslant 0}\A_{\leqslant i}$ be a $\Z_{\geqslant 0}$-filtered associative $\K$-algebra such that  the associated graded algebra $A:=\gr\A$ is commutative and finitely generated. Let $d$ be a positive integer such that $[\A_{\leqslant i},\A_{\leqslant j}]\subset \A_{\leqslant i+j-d}$. We get a natural degree $-d$ Poisson bracket on $A$. We assume that $X:=\operatorname{Spec}(A)$ has finitely many symplectic leaves, meaning that there is a finite stratification of $X$ by irreducible Poisson subvarieties that are symplectic. These subvarieties are called symplectic leaves.

A classical example is the algebras $U_\lambda$ introduced above. Namely, the algebra $U_\lambda$
inherits a filtration from $U(\g)$. The associated graded $\gr U_\lambda$ is the algebra of functions
on the nilpotent cone $\mathcal{N}\subset \g^*$. The symplectic leaves are the adjoint orbits in $\mathcal{N}$
and so there are finitely many.

Other important examples -- quantizations of symplectic resolutions, quantum Hamiltonian reductions
and Symplectic reflection algebras will be introduced in Section \ref{S_alg}.

\subsection{Main results}
Let $\A$ be as in Section \ref{SS_alg} (or, more generally, as described in Section \ref{SS_gen_alg} below).
Let $M$ be a finitely generated left $\A$-module and $\I\subset\A$ be its annihilator. To $M$ and $\A/\I$
we can assign their associated varieties $\VA(M), \VA(\A/\I)$. The former is defined as follows:
we equip $M$ with a {\it good} $\A$-module filtration (``good'' means that $\gr M$ is a finitely generated
$A$-module). The support in $X$ of the $A$-module $\gr M$ is independent of the choice of a good filtration,
we take this support for $\VA(M)$. To define $\VA(\A/\I)$ we view $\A/\I$ as a left $\A$-module. Note that
$\VA(\A/\I)$ is a Poisson subvariety in $X$ and hence is a union of symplectic leaves. In particular, the dimension
of $\VA(\A/\I)$ is even. Note also that $\VA(M)\subset \VA(\A/\I)$.

\begin{Thm}\label{Thm:main_ineq}
Let $\A$ be as above. Assume also that the algebraic fundamental group of every leaf in $X$
is finite. Then the following is true.
\begin{enumerate}
\item Suppose that $M$ is simple.  We have $2\dim \VA(M)\geqslant \dim \VA(\A/\I)$. Moreover, if $\mathcal{L}$
is a symplectic leaf such that $\VA(M)\cap \mathcal{L}\subset \VA(M)$ is open, then $\VA(\A/\I)=\overline{\mathcal{L}}$
and the intersection $\VA(M)\cap \mathcal{L}$ is a coisotropic subvariety in $\mathcal{L}$.
\item Suppose that the $\A$-bimodule $\A$ has finite length and $M$ is arbitrary. Then we have $2\dim \VA(M)
\geqslant \dim \VA(\A/\I)$.
\item Suppose $M$ is simple. Then the intersection of any irreducible component of $\VA(M)$ with $\mathcal{L}$ is non-empty.
\end{enumerate}
\end{Thm}

The claim that $\VA(\A/\I)$ is irreducible when $M$ is simple was first proved by Ginzburg, \cite{Ginzburg_Irr}.
In Section \ref{S_inf_leaf}  it will be shown  that even (1)
of Theorem \ref{Thm:main_ineq} does not need to hold when $X$ fails to have finitely many leaves. Note that (1) implies
that the inequality $2\dim \VA(M)\geqslant \dim \VA(\A/\I)$ will hold for all finite length  modules.

Before we state the second main result of this paper, let us give a definition of a holonomic
$\A$-module. We say that a (locally closed) subvariety $Y\subset X$
is {\it isotropic} if the smooth locus of the intersection $Y\cap \mathcal{L}$ is isotropic in the usual sense
(meaning that the symplectic form vanishes on that  locus)
for any symplectic leaf $\mathcal{L}$ of $X$.  We say that $M$ is {\it holonomic} if $\VA(M)$ is isotropic.
We motivate this definition in Section \ref{S_App} that is a joint appendix with Etingof
and also in Sections \ref{SS_RCA_fin},\ref{SS_red_fin}. Note
that every subquotient of a holonomic module is itself holonomic. An extension of two holonomic modules
is holonomic as well.

\begin{Thm}\label{Thm:main_eq}
Assume that the algebraic fundamental group of every symplectic leaf of $X$ is finite.
Suppose that $M$ is holonomic. Then the following is true:
\begin{enumerate}
\item If  $M$ is simple, then $2\dim \VA(M)=\dim \VA(\A/\I)$.
\item If the $\A$-bimodule $\A$ has finite length, then $M$ has finite length.
Consequently, $2\dim \VA(M)=\dim \VA(\A/\I)$.
\item If $M$ is simple, then the variety $\VA(M)$ is equi-dimensional.
\end{enumerate}
\end{Thm}

A result similar to (3) was obtained by Gabber, \cite{Gabber_equi}, (unpublished).
He proved that the associated variety of an irreducible module over the universal
enveloping algebra of a finite dimensional Lie algebra is equi-dimensional.

Finally, let us state some sufficient conditions for $\A$ to satisfy the assumptions in Theorem \ref{Thm:main_ineq}.

\begin{Thm}\label{Thm:fin_length}
Suppose that $\A$ is one of the following algebras:
\begin{enumerate}
\item[(i)] the global sections of a quantization of a conical symplectic resolution (Section \ref{SS_quant_resol}),
\item[(ii)] the quantum Hamiltonian reduction $\Weyl(V)\red_\lambda G$, where $V$
is a symplectic vector space and $G$ is a reductive group (Section \ref{SS_quant_red}),
\item[(iii)] a (spherical) Rational Cherednik algebra (Section \ref{SS_SRA}).
\end{enumerate}
If $\A$ is as in (i) and (iii), then the algebraic
fundamental group of every leaf in $X$ is finite and the $\A$-bimodule $\A$
has finite length. In (ii), the regular $\A$-bimodule has finite length.
\end{Thm}

We do not know if the algebraic fundamental group of every leaf for Hamiltonian reductions is finite, in general.
By \cite{Namikawa_fund}, this is true for varieties with {\it symplectic singularities},
this includes all varieties with symplectic resolutions and also quotients of vector spaces
by finite groups of symplectomorphisms. Our finiteness claim for Hamiltonian reductions will
follow if we check that the normalization of the reduction has symplectic singularities.
Also note that we do not know any examples, where the (algebraic) fundamental groups of the leaves
are infinite.

The paper is organized as follows. In Section \ref{S_alg} we describe algebras to be studied.
In Section \ref{S_ineq} we prove Theorems \ref{Thm:main_ineq},\ref{Thm:main_eq}. In Section
\ref{S_fin_len} we prove Theorem \ref{Thm:fin_length}. In Section \ref{S_App}
(a joint appendix with Etingof) we motivate the definitions of isotropic subvarieties
and of holonomic modules in case of varieties having a symplectic resolution.

{\bf Acknowledgements}. I am  grateful to Pavel Etingof for numerous discussions related to this project,
for many comments on a preliminary version of this paper, for  explaining me the content of
Section \ref{S_inf_leaf} and for his consent to write a joint appendix. I also thank Geordie Williamson
for mentioning results of Gabber, \cite{Gabber_equi},  and Thierry Levasseuer for
explaining these results.  My work
was partially supported by the NSF grant DMS-1161584. The work of P.E. was partially supported
by the NSF grant DMS-1000113.

\section{Algebras of interest}\label{S_alg}
In this section we will describe algebras we are interested in.
\subsection{General setting}\label{SS_gen_alg}
The most general setting for us is as follows, see \cite{Losev_App}.

Let $\A$ be an associative  algebra with unit equipped with an increasing algebra filtration
$\A=\bigcup_{i\geqslant 0}\A_{\leqslant i}$.
Consider the associated graded algebra $A$ of $\A$ and the center $C$ of $A$.
Then $A$ and $C$ are graded  algebras, let $A_i,C_i$ denote the components of degree $i$.

We assume that the algebra $C$ is finitely generated and that there is an integer $d>0$ and a linear embedding $\iota:C\hookrightarrow \A$
satisfying the following conditions
\begin{itemize}
\item[(i)] $\iota(C_i)\subset \A_{\leqslant i}$ for all $i$ and the composition of $\iota$ with the projection
$\A_{\leqslant i}\twoheadrightarrow A_i$ coincides with the inclusion $C_i\hookrightarrow A_i$.
\item[(ii)] $[\iota(C_i),\A_{\leqslant j}]\subset \A_{\leqslant i+j-d}$ for all $i,j$.
\end{itemize}

Under this assumption, the algebra $C$ has a natural Poisson bracket: for $a\in C_i, b\in C_j$ for
$\{a,b\}$ take the image of $[\iota(a),\iota(b)]$ in $A_{i+j-d}$ (this image lies in $C_{i+j-d}$).
The claim that the map $(a,b)\mapsto \{a,b\}$ extends to a Poisson bracket on $C$ is checked directly.
Note that $\{\cdot,\cdot\}$ is independent of the choice of $\iota$.

For example, if $\A$ is {\it almost commutative}, that is $C=A$, then for $d$ one can take the maximal integer
with $[\A_{\leqslant i}, \A_{\leqslant j}]\subset \A_{\leqslant i+j-d}$ for all $i,j$ (and for $\iota$ the direct sum of sections of the projections $\A_{\leqslant i}\twoheadrightarrow A_i$).

Then we assume that the reduced affine scheme $X$ defined by $C$ has finitely many symplectic leaves. In the next three sections we will provide several concrete examples.

\subsection{Quantizations of symplectic resolutions}\label{SS_quant_resol}
Assume that $X$ admits a conical symplectic resolution $\tilde{X}$. This means that $\tilde{X}$
is a smooth variety equipped with a symplectic form $\omega$, a $\K^\times$-action rescaling $\omega$
with weight $d\in \Z_{>0}$, i.e., $t.\omega=t^d\omega$, and a $\C^\times$-equivariant projective
resolution of singularities
morphism $\rho:\tilde{X} \twoheadrightarrow X$ that is also a Poisson morphism. It is known (due to Kaledin, \cite{Kaledin}) that in this case $X$ has finitely many leaves.

We have a notion of a filtered quantization of $\tilde{X}$, this is a sheaf of filtered algebras
on $\tilde{X}$ in a so called conical topology. Filtered quantizations of $\tilde{X}$ are parameterized, up to an isomorphism, by $H^2_{DR}(\tilde{X})$, see \cite[Section 2]{quant}, let $\mathcal{D}_\lambda$ denote the quantization
corresponding to $\lambda$. Then the global sections algebra $\A_\lambda:=\Gamma(\mathcal{D}_\lambda)$
equipped with the filtration inherited from $\mathcal{D}_\lambda$  satisfies $\gr \A_\lambda=\K[X]$.
This algebra is independent of the choice of a symplectic resolution, see \cite[Section 3.3]{BPW}.

Note that the algebra $U_\lambda$ considered in Section \ref{SS_Bernstein_ineq} is of this form, in this case $\tilde{X}=T^*\mathcal{B}$, where $\mathcal{B}$
is the flag variety of $G$. Other examples include (the central reductions of) finite W-algebras
and also some of quantum Hamiltonian  reductions (such as quantized Nakajima quiver varieties or
quantized hypertoric varieties).

Below we will need some structure theory of conical symplectic resolutions
due to Kaledin and Namikawa gathered in \cite[Section 2]{BPW} and also of their quantizations.
First of all, set $\param:=H^2_{DR}(\tilde{X})$. These spaces are naturally identified for different resolutions
and so depend only on $X$.

Then we have a symplectic scheme $\widehat{X}$ over $\param$ whose fiber
at $0$ is $\tilde{X}$. Outside of the union $S$ of finitely many codimension $1$ vector subspaces in $\param$ the fiber of
$\widehat{X}\rightarrow \param$ is smooth affine. Finally, $\widehat{X}$ comes with a $\K^\times$-action
that
\begin{enumerate}
\item induces the dilation action on $\param$,
\item rescales the fiberwise symplectic form by $t\mapsto t^d$,
\item restricts to  the $\K^\times$-action on $\tilde{X}$
mentioned in the beginning of the section.
\end{enumerate}

The Chern class map $\operatorname{Pic}(\tilde{X})\rightarrow H^2_{DR}(\tilde{X})$  identifies
$\param$ with $\K\otimes_{\Z}\operatorname{Pic}(\tilde{X})$. The image of $\operatorname{Pic}(\tilde{X})$
in $\param$ will be denoted by $\param_{\Z}$, this lattice is also independent of $\tilde{X}$. All codimension $1$ subspaces in $S$ are  rational.

There is a crystallographic group $W$ acting on $\param$ preserving $\param_{\Z}$ and $S$
in such a way that the closure of the movable cone
of $\tilde{X}$ (that is independent of the choice of $\tilde{X}$) is a fundamental chamber for $W$, let us denote it by $\mathcal{C}$. The walls of $\mathcal{C}$ belong to $S$. Moreover, the possible ample cones of conical symplectic resolutions are precisely the chambers cut by $S$ inside of $\mathcal{C}$. The resolution is uniquely recovered
from the chamber.  So for $\chi\in \param_{\Z}\setminus S$ we have the resolution
$\tilde{X}^{\chi}$ corresponding to the ample cone containing $W\chi\cap \mathcal{C}$.

Now let us discuss various isomorphisms between the algebras $\A_\lambda$ and related algebras.
First of all, $\A_\lambda\cong \A_{w\lambda}$ for any $w\in W$  (a filtration preserving isomorphism
that is the identity on the associated graded algebra $\C[X]$), see \cite[Section 3.3]{BPW}. Also
$\A_{\lambda}^{opp}\cong \A_{-\lambda}$, this follows from \cite[Section 2.3]{quant}.

Now let us discuss the localization theorem, see \cite[Section 5]{BPW}. Pick $\chi\in \param_{\Z}\setminus S$
and let $\tilde{X}:=\tilde{X}^\chi$. Let $w\in W$ be such that $W\chi\in \mathcal{C}$.
Consider the category $\operatorname{Coh}(\mathcal{D}_\lambda)$ of all {\it coherent}  $\mathcal{D}_\lambda$-modules,
i.e., sheaves of modules that can be equipped with a filtration whose associated graded is a
coherent sheaf. We have the global section functor $\Gamma_\lambda: \operatorname{Coh}(\mathcal{D}_\lambda)
\rightarrow \A_\lambda\operatorname{-mod}$ and its left adjoint, the localization functor
$\Loc_\lambda: \A_\lambda\operatorname{-mod}\rightarrow \operatorname{Coh}(\mathcal{D}_\lambda)$.
We say that localization holds for $\lambda$ and $\chi$ if $\Gamma_{w\lambda}$ and $\Loc_{w\lambda}$ are mutually
quasi-inverse equivalences.

Here is a basic result regarding parameters where the localization holds, \cite[Section 5]{BPW}. Fix $\lambda$ and $\chi$
as before. Then for $N\in \Z, N\gg 0$, localization holds for $(\lambda+N\chi,\chi)$. Moreover, we can provide an alternative description of the functors $\Gamma_\lambda, \Loc_\lambda$. Let $\mathcal{O}_{N\chi}$ denote the ample
line bundle corresponding to $N\chi$. Then $\mathcal{O}_{N\chi}$ admits a unique quantization to a
$\mathcal{D}_{\lambda+N\chi}$-$\mathcal{D}_\lambda$-bimodule to be denoted by $\mathcal{D}_{\lambda, N\chi}$,
see \cite[Section 5.1]{BPW}.
We write $\A_{\lambda,N\chi}$ for the global sections of $\mathcal{D}_{\lambda,N\chi}$. We note that
the functor $\mathcal{D}_{\lambda,N\chi}\otimes_{\mathcal{D}_\lambda}\bullet$ is an equivalence
$\operatorname{Coh}(\mathcal{D}_\lambda)\xrightarrow{\sim} \operatorname{Coh}(\mathcal{D}_{\lambda+N\chi})$
with quasi-inverse $\mathcal{D}_{\lambda+N\chi,-N\chi}\otimes_{\mathcal{D}_{\lambda+N\chi}}\bullet$.
So we can identify $\operatorname{Coh}(\mathcal{D}_\lambda)$ with $\A_{\lambda+N\chi}\operatorname{-mod}$
by means of $\Gamma(\mathcal{D}_{\lambda,N\chi}\otimes_{\mathcal{D}_\lambda}\bullet)$.

The following lemma is essentially in \cite[Section 5.3]{BPW}.

\begin{Lem}\label{Lem:loc_fun}
Under the identification $\operatorname{Coh}(\mathcal{D}_\lambda)\xrightarrow{\sim}
\A_{\lambda+N\chi}\operatorname{-mod}$, the functor $\Loc_\lambda$ becomes
$\A_{\lambda,N\chi}\otimes_{\A_\lambda}\bullet$ and hence the functor $\Gamma_\lambda$
becomes  $\operatorname{Hom}_{\A_{\lambda+N\chi}}(\A_{\lambda,N\chi},\bullet)$.
\end{Lem}

\subsection{Quantum Hamiltonian reductions}\label{SS_quant_red}
Let $V$ be a symplectic vector space and $G$ a reductive algebraic group acting on $V$
by linear symplectomorphisms. We write $\g$ for the Lie algebra of $G$.

The action of $G$ on $V$ has a natural moment map $\mu:V\rightarrow \g^*$. By the (classical)
Hamiltonian reduction $V\red_0 G$ we mean the categorical quotient $\mu^{-1}(0)\quo G$,
an affine scheme. There are finitely many symplectic leaves in the reduced scheme.
To see this, let $\pi:\mu^{-1}(0)\twoheadrightarrow \mu^{-1}(0)\quo G$ denote the quotient
morphism. For a reductive subgroup $H\subset G$, let $(V\red_0 G)_H$ denote the image of
the locally closed subvariety in $\mu^{-1}(0)$ consisting of all points $v\in \mu^{-1}(0)$
with $G_v=H$ and $Gv$ closed (that this subvariety is locally closed follows from the Luna
slice theorem). That theorem also implies that there are only finitely many  subgroups
$H$ (up to conjugacy) such that $(V\red_0 G)_H$ is nonempty.  Moreover, one can describe the structure of the
formal neighborhood $(V\red_0 G)^{\wedge_x}$ of a point $x\in (V\red_0 G)_H$.
Namely, let $v\in V$ be a point with closed $G$-orbit and $G_v=H$ mapping to $x$.
Consider the $H$-module $U=(T_v Gv)^{\perp}/T_v Gv$, where the complement is taken with respect
to the symplectic form. Then we have an isomorphism
\begin{equation}\label{eq:Ham_iso_compl} V^{\wedge_{Gv}}\cong [(T^*G\times U)\red_0 H]^{\wedge_{G/H}}\end{equation}
(the action of $H$ on $T^*G\times U$
is diagonal with the action on $T^*G$ given by right translations) of symplectic formal
schemes with Hamiltonian $G$-actions, compare with \cite[Section 4]{CB_norm}. (\ref{eq:Ham_iso_compl}) yields an isomorphism of Poisson
formal schemes
\begin{equation}\label{eq:iso_compl}
(V\red_0 G)^{\wedge_x}\cong (U\red_0 H)^{\wedge_0}.
\end{equation}
Note that the symplectic leaf containing $0$ in $U\red_0 H$ is $U^H$. 
From here it is easy to deduce that the symplectic leaves of $V\red_0 G$ are precisely the irreducible
components of the subvarieties $\pi((V\red_0 G)_H)$. 
In particular, the number of the leaves
is finite.

Let us proceed to quantum Hamiltonian reductions. We pick a character
$\lambda$ of $\g$. Consider the Weyl algebra $\Weyl(V)$ of $V$, it comes
with a natural filtration by the degree with respect to $V$. The group $G$ acts on $\Weyl(V)$ by filtered algebra automorphisms.
The comoment map $\mu^*: \g\rightarrow \C[V]$ lifts to a quantum comoment map
$\Phi: \g\rightarrow \Weyl(V)$. There are several liftings, we choose one that decomposes
as $\g\rightarrow \mathfrak{sp}(V)\hookrightarrow \Weyl(V)$, where the homomorphism
$\g\rightarrow \mathfrak{sp}(V)$ is induced by the $G$-action on $V$.
Now we set
$$\Weyl(V)\red_\lambda G:=[\Weyl(V)/\Weyl(V)\{\Phi(\xi)-\langle  \lambda,\xi\rangle| \xi\in \g\}]^G.$$
This is a filtered associative algebra.
It is easy to see that $\C[V\red_0 G]\twoheadrightarrow \gr \Weyl(V)\red_\lambda G$.
It follows that $\operatorname{Spec}(\gr \Weyl(V)\red_\lambda G)$ has finitely many
symplectic leaves as well and so the algebra $\Weyl(V)\red_\lambda G$ fits
the setting described in Section \ref{SS_gen_alg}.

\subsection{Symplectic reflection algebras}\label{SS_SRA}
Let $V$ be a symplectic vector space with form $\omega$ and $\Gamma$ be a finite subgroup of $\operatorname{Sp}(V)$.
We can form the smash-product algebra $S(V)\#\Gamma$ that coincides with
$S(V)\otimes \C\Gamma$ as a vector space, while the product is given by
$(f_1\otimes \gamma_1)\cdot (f_2\otimes \gamma_2)=f_1\gamma_1(f_2)\otimes \gamma_1\gamma_2$.
This algebra is graded, the grading is induced from the usual one on $S(V)$.

A symplectic reflection algebra $H_{t,c}$ is a filtered deformation of $S(V)\#\Gamma$.
Here $t,c$ are parameters, $t\in \C$ and $c$ is a $\Gamma$-invariant function
$S\rightarrow \C$, where $S$ is the set of symplectic reflections in $\Gamma$,
i.e., elements $s\in \Gamma$ such that $\operatorname{rk}(s-1)=2$.
In order to define $H_{t,c}$ we need some more notation. Let $\pi_s$ be the projection
$V\twoheadrightarrow \operatorname{im}(s-1)$ along $\ker(s-1)$. We define $\omega_s\in \bigwedge^2 V^*$
by $\omega_s(u,v):=\omega(\pi_s u,\pi_s v)$.

The algebra $H_{t,c}$ is the quotient of $T(V)\#\Gamma$ by the following relations
$$[u,v]=t\omega(u,v)+\sum_{s\in S}c(s)\omega_s(u,v)s.$$
The algebra $H_{t,c}$ inherits a filtration from $T(V)\#\Gamma$. As Etingof and Ginzburg
proved in \cite[Theorem 1.3]{EG}, we have $\operatorname{gr}H_{t,c}=S(V)\#\Gamma$.

Below we only consider the case of $t=1$. In this case, the algebra $H_{1,c}$ fits into
the framework described in Section \ref{SS_gen_alg}. Indeed, we have $\A= H_{1,c},
A=S(V)\#\Gamma, C=S(V)^\Gamma$.   Note that the algebra $H_{1,c}$ is $\Z/2\Z$-graded
and this grading is compatible with the filtration. A $\Z/2\Z$-graded embedding $\iota:C\rightarrow H_{1,c}$
satisfying (i) of Section \ref{SS_gen_alg} (such an embedding, obviously, exists) also satisfies (ii).

Inside $H_{1,c}$ we have a so called spherical subalgebra $eH_{1,c}e$ that is a quantization
of $S(V)^\Gamma$. Here $e\in \C\Gamma$ is the averaging idempotent.

There are two important special cases of the groups $\Gamma$. First, let $\Gamma_1$ denote
a finite subgroup of $\operatorname{SL}_2(\C)$. Form the group $\Gamma_n:=\mathfrak{S}_n\ltimes
\Gamma_1^n$, where we write $\mathfrak{S}_n$ for the symmetric group in $n$ letters.
This group naturally acts on $\C^{2n}(=(\C^2)^{\oplus n})$ by linear symplectomorphisms.
It turns out that the corresponding symplectic reflection algebra $eH_{1,c}e$ can be
presented as a quantum Hamiltonian reduction, \cite{EGGO,quant}.

Another important special case is when $V$ has a $\Gamma$-stable lagrangian subspace to
be denoted by $\h$. In this case $V\cong \h\oplus \h^*$. The corresponding symplectic
reflection algebra $H_{1,c}$ is known as a Rational Cherednik algebra (RCA, for short).
An important feature of this case is the triangular decomposition
$H_{1,c}=S(\h^*)\otimes \C\Gamma\otimes S(\h)$ (an equality of vector spaces).
We consider the category $\mathcal{O}$ of all finitely generated $H_{1,c}$-modules
with locally nilpotent action of $\h$ following \cite{GGOR}. We will need
two properties of this category:
\begin{enumerate}
\item The category $\mathcal{O}$ has enough projectives and finitely many simples. Every object has finite length.
\item All finite dimensional $H_{1,c}$-modules lie in $\mathcal{O}$.
\end{enumerate}
We will need the following formal consequence of these properties.

\begin{Lem}\label{Lem:RCA_fin}
Let $H_{1,c}$ be a Rational Cherednik algebra. Then there are finitely many finite dimensional
 simple modules and the category of finite dimensional $H_{1,c}$-modules
has enough projective objects. Equivalently, there is the minimal two-sided ideal of finite codimension
in $H_{1,c}$.
\end{Lem}

The same result holds for $eH_{1,c}e$ because every finite dimensional $eH_{1,c}e$-module $N$ is of
the form $eM$, where $M:=H_{1,c}e\otimes_{eH_{1,c}e}N$ is a finite dimensional $H_{1,c}$-module.

\subsection{Completions and decomposition}
Recall that $\A$ is as in Section \ref{SS_gen_alg}. We consider the Rees algebra
$\A_\hbar:=\bigoplus_{i\geqslant 0} \A_{\leqslant i}\hbar^i$ so that
$\A_\hbar/(\hbar-1)=\A, \A_\hbar/(\hbar)=A$.
We write $\A_{\hbar^{\pm 1}}$ for $\K((\hbar))\otimes_{\K[\hbar]}\A_\hbar$.

Let $x\in X$. We consider completions at $x$. Let $\m$ be the maximal ideal of $x$ in $C$.
Set $\tilde{\m}:=A\m$, this is a two-sided ideal of finite codimension in $A$.  Let
$\widehat{\m}\subset\A_\hbar$ be  the preimage of $\tilde{\m}$ under the epimorphism
$\A_\hbar\twoheadrightarrow A$. So we can form the completion
$\A_\hbar^{\wedge_x}:=\varprojlim_{n\rightarrow \infty} \A_\hbar/\widehat{\m}^n$.  This
completion is flat as a left and as a right module over $\A_\hbar$, see \cite[Lemma A.2]{Losev_App}.


Let $\mathcal{L}$ denote the symplectic leaf of $X$ containing $x$.
The algebra $\A_\hbar^{\wedge_x}$ decomposes into the completed tensor product,
\begin{equation}\label{eq:decomp}\A_\hbar^{\wedge_x}\cong
\Weyl_\hbar^{\wedge_0}\widehat{\otimes}_{\K[[\hbar]]}\underline{\A}_\hbar,\end{equation} where $\Weyl_\hbar$
is the Weyl algebra of $T_x\mathcal{L}$ (with  relation $u\otimes v-v\otimes u=\hbar^{d}\omega(u,v)$),
see \cite{Losev_App}. We have an embedding $\K[S]\rightarrow \underline{\A}_\hbar/(\hbar)$,
where $S$ is a formal slice to $\mathcal{L}$ in $\operatorname{Spec}(C)$, with central image.

We write $\A_{\hbar^{\pm 1}}^{\wedge_x}$ for $\K((\hbar))\otimes_{\K[[\hbar]]}\A_\hbar^{\wedge_x}$.
We also write $X^{\wedge_x}$ for the formal neighborhood of $x$ in $X$.


\section{Bernstein inequality}\label{S_ineq}
In this section, we prove  Theorems \ref{Thm:main_ineq} and \ref{Thm:main_eq}.

Let us start with introducing some notation. Recall that $M$ denotes an $\A$-module
and $\I\subset \A$ is its annihilator.

The ideal $\I$ is equipped with a filtration restricted
from $\A$ and so we can form the Rees ideal $\I_\hbar$ that is a graded two-sided ideal in $\A_\hbar$.
Also, we pick a good filtration on $M$ and form the Rees module $M_\hbar$ that is a graded
$\A_\hbar$-module whose annihilator is $\I_\hbar$. The notation $\I_{\hbar^{\pm 1}},M_{\hbar^{\pm 1}}$
has a similar meaning to $\A_{\hbar^{\pm 1}}$.

Set  $M_\hbar^{\wedge_x}:=\A_\hbar^{\wedge_x}\otimes_{\A_\hbar}M_\hbar$. Since $\A_{\hbar}^{\wedge_x}$
is flat over $\A_\hbar$, the module $M_\hbar^{\wedge_x}$
coincides with the $\hbar$-adic completion of $M_\hbar$. Also note that
$\I_\hbar^{\wedge_\chi}:=\A_{\hbar}^{\wedge_\chi}\otimes_{\A_\hbar}\I_\hbar$ is a two-sided ideal in
$\A_\hbar^{\wedge_\chi}$ annihilating $M_\hbar^{\wedge_\chi}$.
The notation $\I_{\hbar^{\pm 1}}^{\wedge_x},M_{\hbar^{\pm 1}}^{\wedge_x}$
has a similar meaning  to $\A^{\wedge_x}_{\hbar^{\pm 1}}$.

For a finitely generated $\A_{\hbar^{\pm 1}}^{\wedge_x}$-module $N_{\hbar^{\pm 1}}$ we can define its support
$\VA(N_{\hbar^{\pm 1}})$, a closed subscheme in $X^{\wedge_x}$. For this, we choose
an $\A_\hbar^{\wedge_x}$-lattice $N_\hbar\subset N_{\hbar^{\pm 1}}$. Then we define
$\VA(N_{\hbar^{\pm 1}})$ as the support of the $A^{\wedge_x}$-module $N_{\hbar}/(\hbar)$.
It is a standard observation that $\VA(N_{\hbar^{\pm 1}})$ is well-defined.

\subsection{Technical lemmas on completions}
Throughout this section, the $\A$-module $M$ is simple.
We start with the following lemma.

\begin{Lem}\label{Lem:simplicity}
$M_{\hbar^{\pm 1}}$ is a simple
$\A_{\hbar^{\pm 1}}$-module.
\end{Lem}
\begin{proof}
 The group $\K^\times$ naturally
acts on $\A_{\hbar^{\pm 1}}, M_{\hbar^{\pm 1}}$. Note that $\A_{\hbar^{\pm 1}}^{\K^\times}\cong \A, M_{\hbar^{\pm 1}}^{\K^\times}=M$. Further, $\A_{\hbar^{\pm 1}}=\K((\hbar))\otimes \A, M_{\hbar^{\pm 1}}=\K((\hbar))\otimes M$.
By the Dixmier lemma, $\operatorname{End}_{\A}(M)=\C$. From here, using an induction on $k$,
we deduce that if $v_1,\ldots,v_k$ are linearly independent elements of $M$,  $\A/\operatorname{Ann}_{\A}(v_1,\ldots,v_k)\cong M^{\oplus k}$. This implies that, for any linearly
independent $v_1,\ldots,v_k\in M$ and arbitrary $w_1,\ldots,w_k\in M$, there is $a\in \A$
with $av_i=w_i$ for all $i=1,\ldots,k$. From here we deduce that for any two nonzero elements of
$\K((\hbar))\otimes M$ there is an element of $\K((\hbar))\otimes_{\K}\A$ mapping one into the other.
So the $\K((\hbar))\otimes \A$-module $\K((\hbar))\otimes M$ is simple.
\end{proof}

Let $x$ be a smooth point in $\VA(M)$.
Lemma \ref{Lem:simplicity} has the following important corollary.

\begin{Cor}\label{Cor:compln_quot}
Let $N_{\hbar^{\pm 1}}$ be a quotient of $M_{\hbar^{\pm 1}}^{\wedge_x}$.
Then $\VA(N_{\hbar^{\pm 1}})=\VA(M)\cap X^{\wedge_x}$ and the preimage of
$\Ann_{\A^{\wedge_x}_{\hbar^{\pm 1}}}(N_{\hbar^{\pm 1}})$ in $\A_\hbar$ coincides with $\I_\hbar$.
\end{Cor}
\begin{proof}
We have $\VA(M^{\wedge_x}_{\hbar^{\pm 1}})=\VA(M)\cap X^{\wedge_x}$ because $M^{\wedge_x}_{\hbar}/\hbar M^{\wedge_x}=
(\gr M)^{\wedge_x}$. So $\VA(N_{\hbar^{\pm 1}})\subset \VA(M)\cap X^{\wedge_x}$. Note also that the simple
$\A_{\hbar^{\pm 1}}$-module $M_{\hbar^{\pm 1}}$ generates $M^{\wedge_x}_{\hbar^{\pm 1}}$. It follows that
$M_{\hbar^{\pm 1}}$ embeds into $N_{\hbar^{\pm 1}}$. Let us take an $\A_\hbar^{\wedge_x}$-lattice
$N_\hbar\subset N_{\hbar^{\pm 1}}$. Then $N_\hbar\cap M_{\hbar^{\pm 1}}$ is a $\K[[\hbar]]\otimes_{\K[\hbar]}\A_{\hbar}$-lattice in $M_{\hbar^{\pm 1}}$. It follows that
$\VA(M)\cap X^{\wedge_x}\subset \VA(N_{\hbar^{\pm 1}})$. This completes the proof of the equality
$\VA(N_{\hbar^{\pm 1}})=\VA(M)\cap X^{\wedge_x}$.

Let us proceed to the equality of ideals. Clearly $\I_\hbar$ annihilates $M_{\hbar^{\pm 1}}^{\wedge_x}$
and so the ideal $\I_\hbar$ lies in the preimage. The opposite inclusion follows from
the claim that $M_{\hbar^{\pm 1}}$ generates $N_{\hbar^{\pm 1}}$.
\end{proof}

Let $\mathcal{L}$ be the symplectic leaf containing $x$ so that we have $\VA(M)\cap X^{\wedge_x}\subset \mathcal{L}^{\wedge_x}$. Recall decomposition (\ref{eq:decomp}).

\begin{Lem}\label{Lem:fin_gen}
The following is true.
\begin{enumerate}
\item  The module $M_\hbar^{\wedge_x}$ is finitely generated over $\Weyl_\hbar^{\wedge_0}$.
\item  The associated variety of the $\Weyl_\hbar^{\wedge_0}$-module $M_{\hbar}^{\wedge_x}$
is $\VA(M)\cap X^{\wedge_x}$.
\item Any quotient $N_{\hbar^{\pm 1}}$ of the $\Weyl_{\hbar^{\pm 1}}^{\wedge_0}$-module $M^{\wedge_x}_{\hbar^{\pm 1}}$
satisfies $\VA(N_{\hbar^{\pm 1}})=\VA(M)\cap X^{\wedge_x}$.
\end{enumerate}
\end{Lem}
\begin{proof}
Let us prove (1).
Both $M_{\hbar}^{\wedge_x}$ and $\Weyl_\hbar^{\wedge_0}$ are complete and separated in
the $\hbar$-adic topology. So it is enough to prove that $M_{\hbar}^{\wedge_x}/\hbar M_{\hbar}^{\wedge_x}$
is finitely generated over $\K[\mathcal{L}^{\wedge_x}]$. But
$M_{\hbar}^{\wedge_x}/\hbar M_{\hbar}^{\wedge_x}=(\gr M)^{\wedge_x}$. Our claim follows from the observation
that $\gr M$ is annihilated by  a sufficiently large power of the ideal of $\overline{\mathcal{L}}$
(and so $(\gr M)^{\wedge_x}$ is annihilated by a sufficiently large power of the maximal ideal
of $\K[S]$).

(2) is a direct consequence of (1) and the choice of $x$.

Let us prove (3). Assume the converse: $\VA(N_{\hbar^{\pm 1}})\subsetneq\VA(M)\cap X^{\wedge_x}$.
Pick an element $n\in N_{\hbar}^{\pm 1}$ and let $J_{\hbar^{\pm 1}}\subset \Weyl_{\hbar^{\pm 1}}^{\wedge_0}$
be its annihilator. Then $J_{\hbar^{\pm 1}}M^{\wedge_x}_{\hbar^{\pm 1}}$ is an $\A^{\wedge_x}_{\hbar^{\pm 1}}$-submodule
in $M^{\wedge_x}_{\hbar^{\pm 1}}$ because of (\ref{eq:decomp}).
We have
$$\VA(M^{\wedge_x}_{\hbar^{\pm 1}}/J_{\hbar^{\pm 1}}M^{\wedge_x}_{\hbar^{\pm 1}})\subset \VA(N_{\hbar^{\pm 1}})
\subsetneq \VA(M)\cap X^{\wedge_x}.$$
This contradicts Corollary \ref{Cor:compln_quot}.
\end{proof}

\begin{Cor}\label{Cor:cois}
$2\dim \VA(M)\geqslant \dim \mathcal{L}$.
\end{Cor}
\begin{proof}
This follows from a well-known fact that the support of a finitely generated
$\Weyl^{\wedge_0}_{\hbar^{\pm 1}}$-module is a coisotropic formal subscheme
(a formal version of the Gabber involutivity theorem, \cite{Ga}, an easier
proof due to Knop can be found in \cite[Section 1.2]{Ginzburg_D_mod}).
\end{proof}

\begin{Lem}\label{Lem:quot_Hom}
There is a simple quotient $N^0_{\hbar^{\pm 1}}$ of the $\Weyl_{\hbar^{\pm 1}}^{\wedge_0}$-module
$M^{\wedge_x}_{\hbar^{\pm 1}}$. Moreover, for any such simple quotient, we have
$$\dim_{\K((\hbar))} \Hom_{\Weyl_{\hbar^{\pm 1}}^{\wedge_0}}(M^{\wedge_x}_{\hbar^{\pm 1}}, N^0_{\hbar^{\pm 1}})<\infty.$$
\end{Lem}
\begin{proof}
The first claim follows from the facts that $\Weyl_{\hbar^{\pm 1}}^{\wedge_0}$ is a Noetherian
algebra and $M^{\wedge_x}_{\hbar^{\pm 1}}$ is its finitely generated module. The dimension
of the Hom is bounded above by the generic multiplicity of $M^{\wedge_x}_{\hbar^{\pm 1}}$
on $\VA(M)\cap X^{\wedge_x}$, this follows from (3) of Lemma \ref{Lem:fin_gen}.
Recall that, by definition, the generic multiplicity of interest is the generic rank of
$M^{\wedge_x}_\hbar/\hbar M^{\wedge_x}$ on $\VA(M)\cap X^{\wedge_x}$.
\end{proof}

Now let $\underline{M}_{\hbar^{\pm 1}}$ denote the intersection of the kernels of all
$\Weyl_{\hbar^{\pm 1}}^{\wedge_0}$-linear homomorphisms $M^{\wedge_x}_{\hbar^{\pm 1}}\rightarrow
N^0_{\hbar^{\pm 1}}$. Note that, again thanks to (\ref{eq:decomp}),
the submodule $\underline{M}_{\hbar^{\pm 1}}$ is actually $\A^{\wedge_x}_{\hbar^{\pm 1}}$-stable.

\begin{Cor}\label{Cor:main}
The quotient $N'_{\hbar^{\pm 1}}:=M^{\wedge_x}_{\hbar^{\pm 1}}/\underline{M}_{\hbar^{\pm 1}}$
is $\Weyl_{\hbar^{\pm 1}}^{\wedge_0}\otimes_{\K((\hbar))}\underline{\A}_{\hbar^{\pm 1}}$-equivariantly
isomorphic to $N^0_{\hbar^{\pm 1}}\otimes_{\K((\hbar))}N^1_{\hbar^{\pm 1}}$, where $N^1_{\hbar^{\pm 1}}$
is a finite dimensional $\underline{\A}_{\hbar^{\pm 1}}$-module.
\end{Cor}

This is a direct corollary of Lemma \ref{Lem:quot_Hom} and the subsequent discussion.

\subsection{Maps between sets of ideals}
Our goal here is to construct a map from a certain set of two-sided ideals in $\A_\hbar^{\wedge_x}$
to the set of two-sided ideals in $\A$ and establish some properties of this map.

Note that $\A_\hbar^{\wedge_x}$ comes equipped with the derivation to be denoted by $\mathsf{eu}$
extended  from the derivation on $\A_\hbar$ coming from the grading. In particular,
we have $\mathsf{eu}\hbar=\hbar$.  The set of two-sided ideals in $\A_\hbar^{\wedge_x}$ we need  consists of
all $\mathsf{eu}$-stable $\hbar$-saturated ideals. We say that a two-sided ideal $\mathcal{J}_\hbar
\subset \A_\hbar^{\wedge_x}$ is $\hbar$-saturated if $\hbar a\in \mathcal{J}_\hbar$
implies $a\in \mathcal{J}_\hbar$. Equivalently, $\hbar$-saturated ideals are precisely the intersections
of two-sided ideals in $\A^{\wedge_x}_{\hbar^{\pm 1}}$ with $\A_\hbar^{\wedge_x}$.
The set of $\hbar$-saturated $\mathsf{eu}$-stable two-sided ideals in $\A_\hbar^{\wedge_x}$
will be denoted by $\mathfrak{Id}(\A_\hbar^{\wedge_x})$.

Let us provide an example of an element in $\mathfrak{Id}(\A_\hbar^{\wedge_x})$ we are interested in.

\begin{Lem}\label{Lem:Euler_stab}
There is a quotient $N_{\hbar^{\pm 1}}$ of the $\A_\hbar^{\wedge_x}$-module
$N'_{\hbar^{\pm 1}}$ (from Corollary \ref{Cor:main})  whose annihilator
$\mathcal{J}_\hbar$ in $\A_\hbar^{\wedge_x}$ lies in $\mathfrak{Id}(\A_\hbar^{\wedge_x})$.
\end{Lem}
\begin{proof}
From the construction, $\mathcal{J}_\hbar$ is $\hbar$-saturated. We need to prove that it
is $\mathsf{eu}$-stable.   We are interested in the structure of maximal $\hbar$-saturated
2-sided ideals $\mathcal{J}'_\hbar\subset \A_\hbar^{\wedge_x}$ such that
$\VA(\A_\hbar^{\wedge_x}/\mathcal{J}'_\hbar)=\mathcal{L}^{\wedge_x}$.
Note that every $\hbar$-saturated two-sided ideal in $\A_\hbar^{\wedge_x}$ is of the form $\Weyl_\hbar^{\wedge_0}\widehat{\otimes}_{\K[[\hbar]]}\underline{\J}_\hbar$,
where $\underline{\J}_\hbar$ is an $\hbar$-saturated two-sided ideal in $\underline{\A}_\hbar$,
see the proof of \cite[Lemma 3.4.3]{Wquant}. We have $\mathcal{J}'_\hbar=\Weyl_\hbar^{\wedge_0}\widehat{\otimes}_{\K[[\hbar]]}\underline{\mathcal{J}}'_\hbar$,
where $\underline{\mathcal{J}}'_\hbar$ is a maximal ideal in $\underline{\A}_\hbar$ such that
$\underline{\A}_\hbar/\underline{\mathcal{J}}'_\hbar$ has finite rank over $\K[[\hbar]]$.
Such ideals $\underline{\mathcal{J}}'_\hbar$ are precisely the annihilators in $\underline{\A}_\hbar$ of
simple finite dimensional $\underline{\A}_{\hbar^{\pm 1}}$-modules. By \cite[Corollary 3.19]{ES},
the algebra $\underline{\A}_{\hbar^{\pm 1}}$ has finitely many finite dimensional irreducible modules.
So let $\underline{\J}^1_\hbar,\ldots, \underline{\J}^m_\hbar$ be their annihilators in
$\underline{\A}_\hbar$ and let $\J^1_\hbar,\ldots,\J^m_\hbar$ be the corresponding
two-sided ideals in $\A_\hbar^{\wedge_x}$. Note that we can take at least one of
the ideals $\J^i_\hbar$ for $\J_\hbar$ (the corresponding $\underline{\J}^i_\hbar$
is the annihilator of a simple quotient of  $N^1_{\hbar^{\pm 1}}$ from Corollary \ref{Cor:main}).

Now we are going to prove that all $\J^i_\hbar, i=1,\ldots,m,$ are $\mathsf{eu}$-stable.
Pick $a\in \K$ and set $\varphi_a:=\exp(a\hbar \mathsf{eu})$, this is a well-defined
automorphism of the $\K$-algebra $\A_\hbar^{\wedge_x}$. We have $\varphi_a(\hbar)\equiv \hbar \mod \hbar^2$.
It follows that $\varphi_a (\J^i_\hbar)$ is still an $\hbar$-saturated ideal of finite corank.
So $\varphi_a (\J^i_\hbar)$ is one of $\J^1_\hbar,\ldots,\J^m_\hbar$. We conclude that
$\varphi_a (\J^i_\hbar)=\J^i_\hbar$ for any $a$. Note also that any left ideal in $\A_\hbar^{\wedge_x}$
is closed, see \cite[Lemma 2.4.4]{HC}. From here and $\varphi_a (\J^i_\hbar)=\J^i_\hbar$ for all $a$ we
conclude that $\J^i_\hbar$ is stable with respect to $\hbar \mathsf{eu}$ and hence also with respect to $\mathsf{eu}$.
\end{proof}

Let us produce a map $\mathfrak{Id}(\A^{\wedge_x}_{\hbar})\rightarrow \mathfrak{Id}(\A)$,
where the target set is that of all two-sided ideals in $\A$. Pick $\J'_\hbar\in
\mathfrak{Id}(\A^{\wedge_x}_{\hbar})$. Let $\I'_\hbar$ be the preimage of $\J'_\hbar$
in $\A_\hbar$. Set $\J'^{\dagger,x}_\hbar:=\I'_\hbar/(\hbar-1)\I'_\hbar$, this is a two-sided
ideal in $\A$.

The following is our main result regarding the map $\bullet^{\dagger,x}$.

\begin{Prop}\label{Prop:support}
If $\VA(\A^{\wedge_x}_\hbar/\J'_{\hbar})=\mathcal{L}^{\wedge_x}$, then $\VA(\A/\J'^{\dagger,x}_\hbar)=\overline{\mathcal{L}}$.
\end{Prop}
\begin{proof}
By a {\it Poisson subbimodule} of $\A^{\wedge_x}_\hbar/\J'_\hbar$ we mean
an $\A_\hbar$-subbimodule $\B'_\hbar$ that is closed with respect
to the operators $\frac{1}{\hbar^d}[\iota(c),\bullet]$ for all $c\in C$, where,
for $c\in C_i$, we view $\iota(c)\in \A_{\leqslant _i}$ as an element
of $\A_\hbar$ via the natural inclusion $\A_{\leqslant i}\hookrightarrow \A_{\hbar}$.

Consider the sum $\B_\hbar$ of all  Poisson subbimodules $\B'_\hbar\subset\A^{\wedge_x}_\hbar/\J'_\hbar$ with the following
properties:
\begin{itemize}
\item $\B'_\hbar$ is a finitely generated $\A_\hbar$-bimodule.
\item $\B'_\hbar$ is $\mathsf{eu}$-stable and the action of $\mathsf{eu}$ on $\B'_\hbar$ is locally finite.
\end{itemize}
Clearly, $\A_\hbar/\I'_\hbar$ provides an example of such a sub-bimodule so $\A_\hbar/\I'_\hbar
\subset \B_\hbar$.
Note that $\B_\hbar$ is itself a Poisson subbimodule. Moreover, an element $a\in \A^{\wedge_x}_\hbar/\J'_\hbar$
lies in $\B_\hbar$ if and only if $a$ is locally finite with respect to $\mathsf{eu}$ and
the Poisson subbimodule generated by $a$ is finitely generated as an $\A_\hbar$-bimodule.
Therefore $\B_\hbar\subset\A^{\wedge_x}_\hbar/\J'_{\hbar}$ is an $\hbar$-saturated
subbimodule. So if $\B_\hbar$ is finitely generated (as a bimodule or as a left/right module,
this does not matter), then the support of $\B_\hbar/\hbar \B_\hbar$ is $\overline{\mathcal{L}}$.

The claim that $\mathcal{B}_\hbar$ is finitely generated follows from
the next lemma,
compare to 
\cite[Proposition 3.7.2]{sraco} or \cite[Lemma 3.3.3]{HC}.
\end{proof}

\begin{Lem}
Suppose that $\pi_1^{alg}(\mathcal{L})$ is finite. If $\mathcal{M}$ is a  Poisson $\C[\Leaf]^{\wedge_x}$-module of finite rank equipped with an Euler
derivation, then the maximal Poisson $\C[\Leaf]$-submodule of $\mathcal{M}$ that is the sum of its finitely
generated Poisson $\C[\overline{\Leaf}]$-submodules (with locally finite action of the  Euler derivation) is
finitely generated.
\end{Lem}
By an Euler derivation, we mean an endomorphism $\mathsf{eu}$ of $\mathcal{M}$ such that
\begin{itemize}
\item $\mathsf{eu}(am)=(\mathsf{eu}a)m+a(\mathsf{eu}m)$,
\item $\mathsf{eu}\{a,m\}=\{\mathsf{eu}a, m\}+\{a, \mathsf{eu}m\}-d\{a,m\}$.
\end{itemize}
Here by $\mathsf{eu}$ on $\C[\Leaf]^{\wedge_x}$ we mean the derivation induced by the contracting $\C^\times$-action
on $X$.

The proof of the lemma first appeared in \cite[Proposition 6.1]{Gies} but we provide it here for the readers
convenience.

\begin{proof}
{\it Step 1}. 
Let $\widetilde{\Leaf}$ be the  Galois covering of $\Leaf$ corresponding to $\pi_1^{alg}(\Leaf)$.
Being the integral closure of
$\C[\overline{\Leaf}]$ in a finite extension of $\C(\Leaf)$, the algebra $\C[\widetilde{\Leaf}]$
is finite over $\C[\overline{\Leaf}]$. The group $\pi_1(\widetilde{\Leaf})$ has no homomorphisms to
any $\operatorname{GL}_m$
by the choice of $\widetilde{\Leaf}$. Also let us note that the $\C^\times$-action on $\Leaf$
lifts to a $\C^\times$-action on $\widetilde{\Leaf}$ possibly after replacing $\C^\times$ with
some covering torus. We remark that the action produces a positive grading on $\C[\widetilde{\Leaf}]$.

{\it Step 2}. Let $\mathcal{V}$ be a weakly $\C^\times$-equivariant finitely generated $D_{\widetilde{\Leaf}}$-module.
We claim that $\mathcal{V}$ is the sum of several copies of $\mathcal{O}_{\widetilde{\Leaf}}$. Indeed,
this is so in the analytic category: $\mathcal{V}^{an}:=\mathcal{O}^{an}_{\widetilde{\Leaf}}\otimes_{\mathcal{O}_{\widetilde{\Leaf}}}\mathcal{V}\cong \mathcal{O}_{\widetilde{\Leaf}}^{an}\otimes\mathcal{V}^{fl}$ (where the superscript ``fl'' means flat
analytic sections)
because  $\pi^{alg}_1(\widetilde{\Leaf})=\{1\}$.
But then the space $\mathcal{V}^{fl}$ carries a holomorphic $\C^\times$-action that has to be diagonalizable
and by characters. So we have an embedding $\Gamma(\mathcal{V})\hookrightarrow \Gamma(\mathcal{V}^{an})^{\C^\times-fin}$
(the superscript means the subspace of all vector lying in a finite dimensional $\C^\times$-stable subspace).
Since $\operatorname{Spec}(\C[\widetilde{\Leaf}])$ is normal, any analytic function on
$\widetilde{\Leaf}$ extends to $\operatorname{Spec}(\C[\widetilde{\Leaf}])$. Since the grading on $\C[\widetilde{\Leaf}]$ is positive, any holomorphic $\C^\times$-semiinvariant function must be polynomial. So the embedding above
reduces to $\Gamma(\mathcal{V})\hookrightarrow \C[\widetilde{\Leaf}]\otimes \mathcal{V}^{fl}$.
The generic rank of $\Gamma(\mathcal{V})$ coincides with the rank of $\mathcal{V}$ that
equals $\dim \mathcal{V}^{fl}$. The embedding $\Gamma(\mathcal{V})\hookrightarrow \C[\widetilde{\Leaf}]\otimes \mathcal{V}^{fl}$ localizes to an embedding $\mathcal{V}\hookrightarrow \mathcal{O}_{\widetilde{\Leaf}}\otimes
\mathcal{V}^{fl}$ of $D_{\widetilde{\Leaf}}$-modules. Since these modules have equal ranks, the embedding
is an isomorphism.

{\it Step 3}. Let $Y$ be a symplectic variety.  We claim that a Poisson $\mathcal{O}_{Y}$-module
carries a canonical structure of a $D_{Y}$-module and vice versa. If $\mathcal{N}$ is a $D_Y$-module,
then we equip it with the structure of a Poisson module via $\{f,n\}:=v(f)n$. Here $f,n$ are local
sections of $\mathcal{O}_Y, \mathcal{N}$, respectively, and $v(f)$ is the skew-gradient of $f$,
a vector field on $Y$. Let us now equip a Poisson module with a canonical D-module structure.
It is enough to do this locally, so we may assume that there is an \'{e}tale map $Y\rightarrow \C^k$.
Let $f_1,\ldots,f_k$ be the corresponding \'{e}tale coordinates. Then we set $v(f)n:=\{f,n\}$. This defines
a D-module structure on $\mathcal{N}$ that is easily seen to be independent of the choice of an
\'{e}tale chart.

Let us remark that a weakly $\C^\times$-equivariant Poisson module gives rise to a weakly $\C^\times$-equivariant
D-module and vice versa.

So the conclusion of the previous three steps is that every weakly $\C^\times$-equivariant Poisson $\mathcal{O}_{\widetilde{\Leaf}}$-module is the direct sum of several copies of $\mathcal{O}_{\widetilde{\Leaf}}$.

{\it Step 4}. Pick a point $\tilde{x}\in \widetilde{\Leaf}$ lying over $x$ so that $\widetilde{\Leaf}^{\wedge_{\tilde{x}}}$
is naturally identified with $\Leaf^{\wedge_x}$. Any finitely generated Poisson module $N$ over $\widetilde{\Leaf}^{\wedge_{\tilde{x}}}$
splits as $\C[\widetilde{\Leaf}]^{\wedge_{\tilde{x}}}\otimes N^0$, where $N^0$
is the Poisson annihilator of $\C[\widetilde{\Leaf}]^{\wedge_{\tilde{x}}}$ in $N$.

Note that the sum of all finitely generated Poisson $\C[\overline{\mathcal{L}}]$-submodules
in $\C[L]^{\wedge_x}$ is a Poisson $\C[\widetilde{\mathcal{L}}]$-submodule because $\C[\widetilde{\mathcal{L}}]$
is finite over $\C[\overline{\mathcal{L}}]$.
 So the claim in the beginning
of the proof will follow if we check that any finitely generated Poisson $\C[\widetilde{\Leaf}]$-module in
$\C[\widetilde{\Leaf}]^{\wedge_{\tilde{x}}}$ with locally finite $\mathsf{eu}$-action coincides with
$\C[\widetilde{\Leaf}]$. For this, let us note that the Poisson center of $\C[\widetilde{\Leaf}]^{\wedge_{\tilde{x}}}$
coincides with $\C$. On the other hand, any finitely generated Poisson submodule with locally finite
action of $\mathsf{eu}$ is the sum of  weakly $\C^\times$-equivariant Poisson submodules. The latter
have to be isomorphic to $\C[\widetilde{\Leaf}]$ by Steps 2,3
and so are generated by the Poisson central elements. This implies our claim and completes the proof.
\end{proof}

\subsection{Proof of main results}
\begin{proof}[Proof of Theorem \ref{Thm:main_ineq}]
Let us prove (1). Let $\J_\hbar$ be as in Lemma \ref{Lem:Euler_stab}. By Corollary \ref{Cor:compln_quot},
$\I_\hbar=\J_\hbar^{\dagger,x}$. Proposition \ref{Prop:support} implies $\VA(\A/\I)=\overline{\mathcal{L}}$.
All other claims in (1) follow from here.

Let us prove (2).
Let $\I'$ be the minimal ideal satisfying $2\dim \VA(M)\geqslant \dim \VA(\A/\I')$ that contains
$\I$, such $\I'$ exists because the $\A$-bimodule $\A$ has finite length. Assume that $\I\neq \I'$. This means
that $\I' M\neq \{0\}$. Let $M_0$ be a maximal proper submodule of $\I' M$. So $\I'(M/M_0)\neq \{0\}$.
Let $\I''$ be the annihilator of $\I' M/M_0$. We have
$2\dim \VA(M)\geqslant 2\dim \VA(\I' M/M_0)\geqslant \dim \VA(\A/\I'')$. It follows that
$2\dim \VA(M)\geqslant \dim \VA(\A/\I'\I'')$. We conclude that $\I'\I''=\I'$. But $\I'\I''$
annihilates $M/M_0$, a contradiction. So (2) is proved.

(3) is a direct consequence of (1).
\end{proof}

\begin{proof}[Proof of Theorem \ref{Thm:main_eq}]
Part (1) is a direct corollary of (1) of Theorem \ref{Thm:main_ineq}.

Let us prove (2). The proof is by induction on $\dim \VA(\A/\I)$. The base, $\dim \VA(\A/\I)=0$,
is trivial: $M$ is finite dimensional. Suppose now that $\dim \VA(\A/\I)=2n$
and that every holonomic module  $M'$ with $\dim \VA(\A/\operatorname{Ann}M')<2n$ has finite length.
Let $\I'$ be the minimal ideal with $\dim \VA(\A/\I')<2n$ containing $\I$, again, this makes sense
thanks to the finiteness condition in (2). By the induction assumption, $M/\I' M$ has finite length.

Similarly to the proof of (2) of Theorem \ref{Thm:main_ineq}, $\I' M$ has a simple
quotient with GK dimension equal to $n$. Let $M_1\subset \I' M$ be the kernel.
If $\dim \VA(\A/\operatorname{Ann}M_1)<2n$, then we are done. If not, then we
can pick $M_2\subset \I' M_1$ such that $\I'M_1/M_2$ is simple with GK dimension equal
to $n$. Then we can produce $M_3$ and so on. By the GK multiplicity reasons, this process terminates.
This establishes a finite Jordan-H\"{o}lder series for $M$ and finishes the proof of the theorem.

(3) is a direct consequence of (3) of Theorem \ref{Thm:main_ineq} and the claim that $\VA(M)$
is a coisotropic subvariety in $\overline{\mathcal{L}}$, compare to the proof of
Corollary \ref{Cor:cois}.
\end{proof}

\subsection{Case of infinitely many leaves}\label{S_inf_leaf}
In this section we show that the Bernstein inequality may fail to hold (even for simple modules) when $X$
 does not have finitely many symplectic leaves. The counter-example was communicated to me by Pavel Etingof.

Let $\A$ be the algebra generated by $x,y,z$ with defining relations $[x,y]=0$, $[z,x]=x$, $[z,y]=1$.
 That is, $\A=\C[z]\ltimes \C[x,y]$, where $z$ acts on $\C[x,y]$ as the vector field
$v:=x\partial_x+\partial_y$. We can put a filtration on $\A$ by assigning degree $1$ to the generators, then $\gr \A=\C[x,y,z]$. Obviously, $\C^3$ cannot have finitely many leaves.

Let $M=\A\otimes _{\C[x,y]}\C$, where $\C$ is the 1-dimensional $\C[x,y]$-module in which
 $x,y$ act by zero. Then $M$ has GK dimension $1$.

We also claim that $\A$ is a simple algebra. To show this, let $\mathcal{J}$ be a nonzero two-sided ideal in $\A$.
Let $m$ be the smallest integer such that $\mathcal{J}$ contains an element $f_0(x,y)z^m+...+f_m(x,y)$
with $f_0\ne 0$. Commuting this element with $y$, we see that $m=0$. Let $\underline{\mathcal{J}}=\mathcal{J}\cap \C[x,y]$; we have seen that $\underline{\mathcal{J}}\neq 0$. Assume that $1\notin \underline{\mathcal{J}}$
(i.e., $\mathcal{J}$ is a proper ideal). Then the zero set $Z(\underline{\mathcal{J}})$ is an algebraic curve.
This curve must be invariant under the vector field $v$. But the integral curves of $v$ have the form
$y=Ce^x$. The only algebraic curve here is $y=0$. We see that $y^m\in \underline{\mathcal{J}}$. But
$\operatorname{ad}(z)^m (y^m)=m!$ and hence $1\in \underline{\mathcal{J}}$, so $\A=\mathcal{J}$.

Thus, $\A$ acts faithfully in $M$ and hence violates the Bernstein inequality ($\A$ has GK dimension 3).
Note that $M$ is simple. Indeed, the subalgebra in $\A$ generated by $y,z$ is isomorphic to
$D(\mathbb{A}^1)$ and, as a module over this subalgebra, $M$ is just $\K[\mathbb{A}^1]$.

\section{Proof of Theorem \ref{Thm:fin_length}}\label{S_fin_len}
The goal of this section is to prove Theorem \ref{Thm:fin_length}. The proofs are different for all
three classes of algebras and will be carried over in the next three sections.

\subsection{Case of Rational Cherednik algebras}\label{SS_RCA_fin}
Recall the notion of a HC bimodule for an RCA $H_{1,c}$ first defined in \cite{BEG},
an equivalent definition that works for all SRA was given in \cite{sraco}. Namely,
we say that a finitely generated $H_{1,c}$-bimodule $\B$ is Harish-Chandra
(shortly, HC) if the adjoint actions of the subalgebras $S(\h)^W, S(\h^*)^W\subset H_{1,c}$
on $\B$ are locally nilpotent. The regular bimodule $H_{1,c}$ gives an example.
Let $\operatorname{HC}_c(\Gamma)$ denote the category of HC $H_{1,c}$-bimodules,
this is an abelian category.

We note that any HC bimodule is finitely generated both as a left and as a right $H_{1,c}$-module.

Note that the symplectic leaves of $X=(\h\oplus \h^*)/\Gamma$ are in one-to-one correspondence
with the conjugacy classes of the parabolic subgroups of $\Gamma$. For a parabolic
subgroup $\underline{\Gamma}\subset \Gamma$, let $\mathcal{L}_{\underline{\Gamma}}$ denote
the corresponding leaf. We note that $\mathcal{L}_{\underline{\Gamma}}\cong \{v\in V| \Gamma_v=\underline{\Gamma}\}/N_{\Gamma}(\underline{\Gamma})$. Recall that the vector space $V$ is  symplectic
so  all subspaces of the form $V^{\underline{\Gamma}}\subset V$ are symplectic as well, in particular,
have even codimension.
Therefore $\{v\in V| \Gamma_v=\underline{\Gamma}\}$ is simply connected, and $\pi_1(\mathcal{L}_{\underline{\Gamma}})=
N_\Gamma(\underline{\Gamma})/\underline{\Gamma}$, in particular, this group is finite.

\begin{Rem}
Let $\pi:V\rightarrow V/\Gamma$ denote the quotient morphism. We note that $Y\subset V/\Gamma$
is isotropic if and only if $\pi^{-1}(Y)$ is isotropic. This may serve as a motivation for
our definition of an isotropic subvariety.
\end{Rem}

We have an exact functor $\bullet_{\dagger,\underline{\Gamma}}: \operatorname{HC}_c(\Gamma)\rightarrow \operatorname{HC}_{\underline{c}}(\underline{\Gamma})$, where  $\underline{c}$ stands for the restriction of
$c$ to $S\cap \underline{\Gamma}$. This functor was introduced and studied  in \cite{sraco}, see
\cite[Section 3.6]{sraco}. It has the following properties:
\begin{enumerate}
\item We have $\B_{\dagger,\underline{\Gamma}}=0$ if and only if $\mathcal{L}_{\underline{\Gamma}}\cap \VA(\B)=\varnothing$.
\item We have that $\B_{\dagger,\underline{\Gamma}}$ is finite dimensional and nonzero if
and only if $\mathcal{L}_{\underline{\Gamma}}$ is an open leaf in $\VA(\B)$.
\item Let $\HC_{c,\underline{\Gamma}}(\Gamma)$ stand for the category of all HC bimodules
$\B$ such that $\dim\B_{\dagger,\underline{\Gamma}}<\infty$. Then there is a right adjoint
functor $\bullet^{\dagger,\underline{\Gamma}}:\HC_{\underline{c},\underline{\Gamma}}(\underline{\Gamma})\rightarrow \HC_{c,\underline{\Gamma}}(\Gamma)$ for $\bullet_{\dagger,\underline{\Gamma}}$. Moreover, the image
of $\bullet^{\dagger,\underline{\Gamma}}$ consists of  HC bimodules supported on
$\overline{\mathcal{L}_{\underline{\Gamma}}}$.
\end{enumerate}
The first two properties follow from the construction, while the last one is proved similarly
to \cite[Proposition 3.7.2]{sraco}.

Now let us order parabolic subgroups $\Gamma_1,\ldots,\Gamma_k$ of $\Gamma$
in such a way that if $\Gamma_i$ is conjugate to a subgroup of $\Gamma_j$
(or, equivalently, if $\mathcal{L}_{\Gamma_i}\subset \overline{\mathcal{L}_{\Gamma_j}}$),
then $i>j$. In particular, $\Gamma_k=\{1\}$ and $\Gamma_1=\Gamma$. Below we will write
$\mathcal{L}_i$ instead of $\mathcal{L}_{\Gamma_i}$.

We are going to prove the following claim using the increasing induction on $i$:
\begin{itemize}
\item[($*_i$)] For any HC bimodule $\B$, there is the minimal subbimodule $\B_i\subset \B$ such that
$\VA(\B/\B_i)\subset \bigcup_{j\leqslant i}\mathcal{L}_i$. The length
of $\B/\B_i$ is finite.
\end{itemize}

We start with the base, $i=0$. Let $\I_0\subset H_{1,c}$ be the minimal ideal of finite codimension
whose existence was established in Lemma \ref{Lem:RCA_fin}. We claim that $\B_0=\B\I_0$.
Indeed, $\B/\B\I_0$ is a finitely generated right $H_{1,c}/\I_0$-module and hence is
finite dimensional.  Further, $\B\I_0$ has no finite dimensional quotients, this is proved similarly to the
proof of (2) of Theorem \ref{Thm:main_ineq}. The equality $\B_0=\B\I_0$ follows.

Now suppose that $(*_{i-1})$ is proved for all HC bimodules $\B'$. In particular,
all HC bimodules supported on $\bigcup_{j<i}\mathcal{L}_{j}$ have finite length.
Let us prove $(*_i)$.

First of all, let us check that any HC bimodule $\B$ has minimal
subbimodule $\B_i$ such  that $\VA(\B/\B_i)\subset \bigcup_{j\leqslant i}\mathcal{L}_j$.
Consider the HC $H_{1,\underline{c}}(\Gamma_i)$-bimodule $\B_{\dagger,\Gamma_i}$. Let
$\underline{\B}$ be its maximal finite dimensional quotient. Then, for $\B_i$,
we take the kernel of the natural map $\B\rightarrow \underline{\B}^{\dagger,\underline{\Gamma}}$.
It is straightforward to see that $\B_i$ has required properties.

Now let us check that all HC bimodules supported on $\bigcup_{j\leqslant i}\mathcal{L}_{j}$
have finite length. By $(*_{i-1})$, any HC bimodule has the maximal quotient supported
on $\bigcup_{j<i}\mathcal{L}_{j}$. Let $\B^1$  be the kernel of such quotient. Then $\B^1$
has a simple quotient $\underline{\B}^1$ with $\mathcal{L}_{i}\subset \VA(\underline{\B}^1)$.
Let $\hat{\B}^1\subset \B^1$ be the kernel of the projection $\B^1\twoheadrightarrow \underline{\B}^1$.
The bimodule $\hat{\B}^1$ again has the maximal quotient supported on
$\bigcup_{j<i}\mathcal{L}_{j}$. Let $\B^2$ denote the kernel of this quotient.
In this way, we get a sequence of subbimodules
$$\B^0=\B\supset \B^1\supsetneq \B^2\supsetneq\ldots.$$
The sequence cannot have more than $\dim \B_{\dagger,\Gamma_i}$ terms. By the construction
and the inductive assumption, all subsequent quotients $\B^i/\B^{i+1}$ have finite length.

This completes the proof of Theorem \ref{Thm:fin_length} for Rational Cherednik algebras $H_{1,c}$.
The proof for the spherical subalgebras is similar.

\subsection{Case of quantum Hamiltonian reductions}\label{SS_red_fin}
Recall that we prove are going to prove that the $\Weyl(V)\red_\lambda G$-bimodule $\Weyl(V)\red_\lambda G$
has finite length. We will start however by proving that
any holonomic $\Weyl(V)\red_\lambda G$-module has finite length.

Consider the category $\Weyl(V)\operatorname{-mod}^{G,\lambda}$ of all $(G,\lambda)$-equivariant
finitely generated $\Weyl(V)$-modules. Recall that being $(G,\lambda)$-equivariant means
that a module is equipped with a rational $G$-action such that the corresponding
$\g$-action is given by $\xi\mapsto \Phi(\xi)-\langle\lambda,\xi\rangle$. We have a quotient functor $\varphi:
\Weyl(V)\operatorname{-mod}^{G,\lambda}\twoheadrightarrow
\Weyl(V)\red_\lambda G\operatorname{-mod}$ given by taking the
$G$-invariants. The left adjoint (and right inverse) functor
$\varphi^*$ is given by
$$N\mapsto (\Weyl(V)/\Weyl(V)\{\Phi(\xi)-\langle \lambda,\xi\rangle\})\otimes_{\Weyl(V)\red_\lambda G}N.$$
Note that $\VA(\varphi^* N)\subset \pi^{-1}(\VA(N))$, where, recall, $\pi$ stands for the quotient
morphism $\mu^{-1}(0)\twoheadrightarrow V\red_0 G$. The claim that $N$ has finite length will follow
if we check that $\varphi^* N$ has finite length. The latter will follow if we check that
$\varphi^* N$ is a holonomic $\Weyl(V)$-module. And in order to establish that, it is sufficient
to prove the following proposition.

\begin{Prop}\label{Prop:isotr_red}
A subvariety $Y\subset V\red_0 G$ is isotropic if and only if $\pi^{-1}(Y)\subset V$ is isotropic.
\end{Prop}
Note that this proposition can be regarded as one of motivations for our definitions of isotropic
subvarieties and holonomic modules.

\begin{proof}
The proof is in two steps.

{\it Step 1.} We start by proving that $\pi^{-1}(0)$ is an isotropic subvariety in $V$.
Let $\tilde{\pi}:V\rightarrow V\quo G$ be the quotient morphism so that $\pi^{-1}(0)=
\tilde{\pi}^{-1}(0)\cap\mu^{-1}(0)$. Recall that $\tilde{\pi}^{-1}(0)$ admits a
stratification by Kempf-Ness strata, see, e.g., \cite[Section 5.6]{PV}. Namely, for a one-parameter subgroup
$\gamma:\K^\times \rightarrow G$, let $V_{>0,\gamma}$ denote the sum of all
$\gamma$-eigenspaces with characters $t\mapsto t^k, k>0$. Then there are finitely many
one-parametric subgroups $\gamma_1,\ldots,\gamma_s:\K^\times \rightarrow G$ such that
$\tilde{\pi}^{-1}(0)=\bigcup_{i=1}^s GV_{\gamma_i,>0}$. So it is enough to check that
$\mu^{-1}(0)\cap GV_{\gamma,>0}$ is an isotropic subvariety of $V$.
Pick $x\in \mu^{-1}(0)\cap V_{\gamma,>0}$. Any  tangent vector to $\mu^{-1}(0)\cap GV_{\gamma,>0}$
has the form $\xi_x+ v$, where $\xi\in \g$ (and $\xi_x$ denotes the velocity vector
defined by $\xi$) and $v\in V_{\gamma,>0}\cap \ker d_x\mu$. So let us take
$\xi^1,\xi^2\in \g$ and $v^1,v^2\in V_{\gamma,>0}\cap \ker d_x\mu$. We need
to check that
$\omega(\xi^1_x+v^1, \xi^2_x+v^2)=0$. We have $\omega(\xi^1_x,\xi^2_x)=0$ because
$x\in \mu^{-1}(0)$ and $\omega(v^1,v^2)=0$ because $V_{\gamma,>0}$ is an isotropic
subspace. We also have $\omega(\xi^1_x,v^2)=0$ because $v^2\in \ker d_x \mu$.
Similarly, $\omega(v^1,\xi^2_x)=0$. We conclude that $\mu^{-1}(0)\cap GV_{\gamma,>0}$
is an isotropic subvariety. Hence $\pi^{-1}(0)$ is an isotropic subvariety.

{\it Step 2}. Now we claim that $Y\subset V\red_0 G$ is an isotropic subvariety
of some symplectic leaf if and only if $\pi^{-1}(Y)$ is an isotropic subvariety of $V$.
It is enough to prove that, for a smooth point $y\in Y$, the formal neighborhood of $Gv$ in $\pi^{-1}(Y)$
is isotropic if and only if the formal neighborhood $Y^{\wedge_y}$ of $y$ in $Y$ is isotropic
inside the symplectic leaf. Here $Gv$ denotes a unique closed orbit in $\pi^{-1}(y)$. Thanks to the slice theorem
recalled in Section \ref{SS_quant_red}, one needs to prove that $Y^{\wedge_y}$ is isotropic
if and only if the preimage of $Y^{\wedge_y}$ in the formal neighborhood of $G/H$ in $(T^*G\times U)\red_0 H$
is isotropic (we use the notation from that section). This preimage is the bundle over $G/H$ with fiber that is the
preimage of $Y^{\wedge_y}\subset (U\red_0 H)^{\wedge_0}$ in $U^{\wedge_0}$. The latter decomposes into the product
of a formal subscheme in $(U^H)^{\wedge_0}$ and the intersection of
$\pi_H^{-1}(0)$ with $(U/U^H)^{\wedge_0}$. This intersection is isotropic by Step 1.
So the preimage of $Y^{\wedge_y}$ is isotropic if and only if $Y^{\wedge_y}$ is so.
This completes the proof of Step 2.
 \end{proof}

\begin{Cor}\label{Cor:red_fin_length}
The regular $\Weyl(V)\red_\lambda G$-bimodule has finite length.
\end{Cor}
\begin{proof}
By \cite[Lemma 3.14]{BPW}, we have $(\Weyl(V)\red_\lambda G)^{opp}\cong \Weyl(V)\red_{-\lambda} G$.
So $(\Weyl(V)\red_\lambda G)\otimes (\Weyl(V)\red_\lambda G)^{opp}\cong \Weyl(V\oplus V)\red_{(\lambda,-\lambda)}G\times G$. Since $\Weyl(V)\red_\lambda G$ is a holonomic module over the latter algebra, it has finite length.
\end{proof}

\subsection{Case of quantized symplectic resolutions}
The algebraic fundamental group of every leaf in $X$ is finite by \cite{Namikawa_fund}.

Now we are going to show that the $\A_\lambda$-bimodule $\A_\lambda$ has finite length.
Pick a symplectic resolution $\tilde{X}=\tilde{X}^{\chi}$ of $X$. Then $\tilde{X}^\chi\times
\tilde{X}^{-\chi}$ is a symplectic resolution of $X\times X$. We see that, for $N\gg 0$,
localization holds for the parameter $(\lambda+N\chi, -\lambda-N\chi)$. We claim the
$\A_{\lambda+N\chi}$-bimodule $\A_{\lambda+N\chi}$ has finite length. Indeed, the localization
of this module to $\tilde{X}^{\chi}\times \tilde{X}^{-\chi}$ is supported on the preimage of the diagonal
$X_{diag}\subset X$. This preimage is isotropic by Lemma \ref{Lem:isotr_resol} below
and hence $\operatorname{Loc}(\A_{\lambda+N\chi})$ has finite length. Since localization
holds for $(\lambda+N\chi, -\lambda-N\chi)$ and $(\chi,-\chi)$ we conclude that the regular
bimodule $\A_{\lambda+N\chi}$ has finite length.

From here we are going to deduce that $\A_\lambda$ has finite length. Let us order the
symplectic leaves $\mathcal{L}_1,\ldots,\mathcal{L}_k$ of $X$ in such a way that
$\mathcal{L}_i\subset \overline{\mathcal{L}_j}$ implies $i\leqslant j$. We are going to
show that there is a minimal ideal $\I_i\subset \A$ with $\VA(\A/\I_i)\subset \bigcup_{j\leqslant i}
\mathcal{L}_j$ and that the length of $\A/\I_i$ is finite. The base $i=-1$ is vacuous.

Let $\I\subset \A_\lambda$ be a two-sided ideal with $\VA(\A_\lambda/\I)\subset\bigcup_{j\leqslant i}\mathcal{L}_j$.
Consider a complex $\A_{\lambda,N\chi}\otimes^L_{\A_\lambda}\A_\lambda/\I$. Its homology are Harish-Chandra (shortly, HC)
$\A_{\lambda+N\chi}$-$\A_\lambda$-bimodules, see \cite[Section 6.1]{BPW}, annihilated by $\I$ on the right.
Recall the definition of a HC $\A_{\lambda+N\chi}$-$\A_\lambda$-bimodule. This is a bimodule $\B$
that can be equipped with a filtration whose associated graded is a finitely generated $\C[X]$-module.

\begin{Lem}\label{Lem:LR_Ann}
Let $\B$ be a HC $\A_{\lambda'}$-$\A_\lambda$-bimodule with left annihiator $\I'$ and right annihilator $\I$.
Then $\VA(\A_{\lambda'}/\I')=\VA(\B)=\VA(\A_\lambda/\I)$.
\end{Lem}
\begin{proof}
First of all, it is clear that $\VA(\B)\subset \VA(\A_{\lambda'}/\I'),\VA(\A_\lambda/\I)$.
Let us prove that $\VA(\A_\lambda/\I)\subset \VA(\B)$, the remaining inclusion is analogous.
Consider $\tilde{\A}:=\operatorname{Hom}_{\A_{\lambda'}}(\B,\B)$. This $\A_\lambda$-bimodule
comes with a filtration induced from that on $\B$: by definition, $\tilde{\A}_{\leqslant i}$
consists of all homomorphisms mapping $\B_{\leqslant j}$ to $\B_{\leqslant i+j}$
for all $j$.

Clearly, $\gr\tilde{\A}\subset
\operatorname{Hom}_{\C[X]}(\gr\B,\gr\B)$. The latter is a finitely generated $\C[X]$-module
supported on $\VA(\B)$. So $\tilde{\A}$ is HC and $\VA(\tilde{\A})\subset \VA(\B)$.
The inclusion $\VA(\A/\I)\subset \VA(\B)$ follows from the natural inclusion $\A/\I\subset
\tilde{\A}$.
\end{proof}

Since the $\A_{\lambda+N\chi}$-bimodule $\A_{\lambda+N\chi}$ has finite length, it follows that
there is a two-sided ideal $\I'\subset \A_{\lambda+N\chi}$ with $\VA(\A_{\lambda+N\chi}/\I')\subset\bigcup_{j\leqslant i}\mathcal{L}_j$ annihilating (from the left) all bimodules $\operatorname{Tor}^p_{\A_\lambda}(\A_{\lambda,N\chi}, \A_\lambda/\I)$.
Let $\underline{\I}$ denote the right annihilator of $\A_{\lambda,N\chi}/\I'\A_{\lambda,N\chi}$.
This is an ideal in $\A_\lambda$ with $\VA(\A_\lambda/\underline{\I})\subset \bigcup_{j\leqslant i}\mathcal{L}_j$
by Lemma \ref{Lem:LR_Ann}. Moreover, $\underline{\I}$ annihilates (from the right) all bimodules
\begin{align*}&\operatorname{Ext}^q_{\A_{\lambda+n\chi}}(\A_{\lambda,N\chi}, \operatorname{Tor}^p_{\A_\lambda}(\A_{\lambda,N\chi}, \A_\lambda/\I))=\\&=\operatorname{Ext}^q_{\A_{\lambda+n\chi}/\I'}(\A_{\lambda+N\chi}/\I'\otimes^L_{\A_{\lambda+N\chi}}\A_{\lambda,N\chi}, \operatorname{Tor}^p_{\A_\lambda}(\A_{\lambda,N\chi}, \A_\lambda/\I)).\end{align*}
But \begin{equation}\label{eq:RHom_LLoc} R\operatorname{Hom}_{\A_{\lambda+N\chi}}(\A_{\lambda, N\chi}, \A_{\lambda,N\chi}\otimes^L_{\A_\lambda}\A_\lambda/\I)=\A_\lambda/\I.\end{equation}
Indeed, by Lemma \ref{Lem:loc_fun}, the left hand side is $R\Gamma_\lambda(L\Loc_\lambda(\A_\lambda/\I))$,
and the latter is isomorphic to $\A_\lambda/\I$. Using (\ref{eq:RHom_LLoc}) together with
the standard spectral sequence for a composition of derived functors,
we see that $\underline{\I}$ annihilates $\A_\lambda/\I$ from the left.
So $\underline{\I}\subset \I$ and hence $\underline{\I}$ is the minimal ideal in $\A_\lambda$ with
$\VA(\A/\underline{\I})=\bigcup_{j\leqslant i}\mathcal{L}_j$.

The proof that  every HC $\A_\lambda$-bimodule $\B$ has finite length is now similar to the RCA
case considered in Section \ref{SS_RCA_fin}. Namely, we prove this statement by induction on $i$, where $\VA(\B)\subset
\bigcup_{j\leqslant i}\mathcal{L}_j$. Let $\I_i\subset \A_\lambda$ be the minimal
ideal with $\VA(\A_\lambda/\I_i)\subset \bigcup_{j\leqslant i}\mathcal{L}_j$.
Then, similarly to Section \ref{SS_RCA_fin}, $\B/\I_i\B$ is the maximal quotient of $\B$ such that $\VA(\B/\I_i\B)\subset
\bigcup_{j\leqslant i}\mathcal{L}_j$. The bimodule $\I_{i-1}\B/\I_i\B$ has a simple
quotient whose associated variety is not contained in $\bigcup_{j<i}\mathcal{L}_j$.
Then we can form a filtration on $\B/\I_i\B$ similarly to the RCA case. This filtration
is finite: the number of terms is bounded by the multiplicity of $\B/\I_i\B$ on $\mathcal{L}_i$.
The successive quotients all have finite length by the inductive assumption.

\begin{Rem}\label{Rem:SRA}
The spherical subalgebra $eH_{1,c}e$ for a group of the form $\Gamma_n$ is an algebra of the form
$\A_\lambda$ (\cite[Section 6]{quant}). In fact, the $H_{1,c}$-bimodule $H_{1,c}$ has finite length.
This follows from the observation that $H_{1,c}$ is the algebra of global sections
of the endomorphism algebra of a suitable locally free right $\mathcal{D}_\lambda$-module
(a quantized Procesi bundle). The arguments above can be repeated for $H_{1,c}$.

There are symplectic reflection groups that are neither $\Gamma_n$ nor come from
complex reflection groups. For them we do not know whether the $H_{1,c}$-bimodule
$H_{1,c}$ has finite length. Clearly, using the argument from Section \ref{SS_RCA_fin},
this can be reduced to showing that $H_{1,c}$ has the minimal ideal of finite codimension.

The question of whether the generalized Bernstein inequality holds for Symplectic reflection
algebras was asked by Etingof and Ginzburg, \cite{EG_problems}, and also by Gordon in \cite{Gordon_review}.
The problem is now completely settled for the two nice classes
of symplectic reflection groups. In the general case, it is settled in the case when
$H_{1,c}$ is simple (obviously, the regular bimodule has finite length in this case).
Recall, \cite[Theorem 1.4.2, Section 4.2]{sraco} that the algebra $H_{1,c}$
is simple outside of countably many hyperplanes.
\end{Rem}

\section{Appendix 1 by P. Etingof and I. Losev}\label{S_App}
In this section we deal with algebras of the form $\A_\lambda$, see Section \ref{SS_quant_resol}.
Our goal is to compare isotropic subvarieties in $X$ and in $\tilde{X}$ and also holonomic
modules for $\A_\lambda$ and for $\mathcal{D}_\lambda$. See Section \ref{SS_quant_resol} for the notation.
\subsection{Isotropic subvarieties}
Let $M$ be a symplectic variety and $Z$ its  subvariety.
It is shown in \cite[Prop. 1.3.30]{Chriss_Ginzburg} that $Z$ is isotropic if and only if, for any smooth locally closed subvariety $W\subset Z$, one has $\omega|_W=0$. This implies that any locally closed subvariety of an isotropic subvariety is isotropic. Also, the closure of an isotropic subvariety and a finite union of closed isotropic subvarieties is isotropic.

Now suppose that $\rho: \tilde{X}\rightarrow X$ is a symplectic resolution of singularities of an
irreducible Poisson variety $X$, and let $\omega$ be the symplectic form on $\tilde{X}$.
Let $\mathcal{L}_i$ be the symplectic leaves of $X$ (recall that there are finitely many), and let $Y_i$ be their preimages in $\tilde{X}$. It is well known that the pullback of the symplectic form on $\mathcal{L}_i$
to a locally closed subvariety $W\subset Y_i$ is the restriction of $\omega$ to $W$.

\begin{Lem}\label{Lem:isotr_resol} (i) If $Z\subset X$ is isotropic, then $\rho(Z)\cap \mathcal{L}_i$ is isotropic for all $i$.

(ii) If $L\subset \mathcal{L}_i$ is isotropic, then $\rho^{-1}(L)\subset \tilde{X}$ is isotropic.
\end{Lem}

\begin{proof} (i) We have $\rho(Z)\cap \mathcal{L}_i=\rho(Z\cap Y_i)$. So without loss of generality
we may assume that $Z\subset Y_i$, and we need to show that $\rho(Z)$ is isotropic
(note that $\rho(Z)$ is a locally closed subvariety of $Y$ since $\rho$ is proper).
Again without loss of generality we may assume that $Z$ is irreducible.
Let $z\in Z$ be a generic smooth point.
The map $d\rho_z: T_zZ\rightarrow T_{\rho(z)}\rho(Z)$ is surjective.
So, for any two tangent vectors $v_1,v_2\in T_{\rho(z)}\rho(Z)$, there are $w_1,w_2\in T_zZ$ such that $d\rho_z(w_i)=v_i$.
Then $\omega_{\mathcal{L}_i}(v_1,v_2)=\omega(w_1,w_2)=0$, since $w_1,w_2$ are tangent to $Z$ and $Z$ is isotropic.
Thus, $\rho(Z)$ is isotropic near a generic point and hence isotropic.

(ii) Let $L\subset \rho^{-1}(\mathcal{L}_i)$ be a smooth locally closed subvariety, and $w_1,w_2$ two tangent vectors to it at a point $z$. Then $\omega_{\mathcal{L}_i}(d\rho_z(w_1),d\rho_z(w_2))=\omega(w_1,w_2)=0$, the last equality
holds because $L$ is isotropic. This implies that $\rho^{-1}(L)$ is isotropic.
\end{proof}

\subsection{Holonomic modules and localization}
Recall that to a coherent $\mathcal{D}_\lambda$-module $\mathcal{M}$ one can assign its support
$\VA(\mathcal{M})$ that, by definition, is the support of $\gr\mathcal{M}$, where the associated
graded is taken with respect to an arbitrary good filtration. We say that $\mathcal{M}$
is holonomic if $\VA(\mathcal{M})$ is lagrangian (it is always coisotropic by the Gabber involutivity theorem,
\cite{Ga}).

The following lemma justifies the definition of a holonomic $\A_\lambda$-module in the case when $X$
admits a conical symplectic resolution of singularities, $\tilde{X}$.

\begin{Lem}\label{Lem:holon_local}
(i) Let $M$ be a holonomic $\A_\lambda$-module. Then all the homology of $L\Loc_\lambda(M)$
are holonomic $\mathcal{D}_\lambda$-modules.

(ii) Let $\mathcal{M}$ be a holonomic $\mathcal{D}_\lambda$-module. Then all the homology
of $R\Gamma_\lambda(M)$ are holonomic $\A_\lambda$-modules.
\end{Lem}
\begin{proof}
(i) It is easy to see that the support of any module $H_i(L\Loc_\lambda(M))$
lies in $\rho^{-1}(\VA(M))$, in particular, by (i) of  Lemma \ref{Lem:isotr_resol}, is isotropic.
So $H_i(L\Loc_\lambda(M))$ is holonomic.

(ii) is proved similarly to (i), once we observe that the homology of $R\Gamma_\lambda(\mathcal{M})$
are supported on $\rho(\VA(\mathcal{M}))$.
\end{proof}

\end{document}